\documentclass[a4paper,12pt]{amsart}

\newcommand{\Cinf}{C^\infty}

\newcommand{\diff}{\mathrm{d}}
\newcommand{\pr}{\mathrm{pr}}

\newcommand{\n}[1]{\left\Vert #1\right\Vert} 
\newcommand{\la}{\left\langle}
\newcommand{\ra}{\right\rangle} 

\newcommand{\R}{\mathbb{R}}

\newcommand{\be}{\begin{equation}}
\newcommand{\ee}{\end{equation}}
\newcommand{\bea}{\begin{eqnarray}}
\newcommand{\eea}{\end{eqnarray}}
\newcommand{\ben}{\begin{displaymath}}
\newcommand{\een}{\end{displaymath}}
\newcommand{\bean}{\begin{eqnarray*}}
\newcommand{\eean}{\end{eqnarray*}}
\newcommand{\mc}[1]{\mathcal{#1}}

\newcommand{\dif}[1]{\frac{\diff}{\diff #1}}

\newcommand{\mscr}[1]{\textrm{\fontencoding{OMS}\fontfamily{ztmcm}\selectfont#1}
}

\usepackage{amsmath}
\usepackage{amssymb}
\usepackage{amsthm}
\usepackage[toc,page]{appendix}
\usepackage[english]{babel}
\usepackage{verbatim}
\usepackage[T1]{fontenc}
\usepackage{fancyhdr}
\usepackage[hmargin=3cm,vmargin=2.5cm]{geometry}
\usepackage{ifthen}
\usepackage{color}

\newcommand{\mcH}{\mathcal{H}}
\newcommand{\FF}{\mathbb{F}}
\newcommand{\Euc}{\mathrm{SE}}

\newcommand{\HH}{\mathbb{H}}
\newcommand{\II}{\mathrm{II}}

\newcommand{\GL}{\mathrm{GL}}

\newcommand{\VF}{\mathrm{VF}}
\newcommand{\SO}{\mathrm{SO}}

\newcommand{\rank}{\mathrm{rank\ }}
\newcommand{\Iso}{\mathrm{Iso}}

\newcommand{\hM}{\hat{M}}
\newcommand{\hg}{\hat{g}}

\newcommand{\RDist}{\mc{D}_{\mathrm{R}}}

\newcommand{\LRD}{\mscr{L}_{\mathrm{R}}}

\def\[#1\]{\begin{align*}#1\end{align*}}

\newtheoremstyle{theorem}{0.5cm}{\topsep}%
   {\sffamily}
   {}
   {\bfseries}
   {}
   {2ex}
   {\thmname{#1}\thmnumber{ #2}\thmnote{ #3}}
\theoremstyle{theorem}
\newtheorem{theorem}{Theorem}[section]

\newtheorem{proposition}[theorem]{Proposition}

\newtheorem{corollary}[theorem]{Corollary}
\newtheorem{remark}[theorem]{Remark}
\newtheorem{definition}[theorem]{Definition}
\newtheorem{lemma}[theorem]{Lemma}

\newcommand{\doo}{\partial}

\title[Non-Euclidean de Rham theorem]{Extension of de Rham decomposition theorem via non-Euclidean development}

\author[Y. Chitour, M. Godoy M., P. Kokkonen]{Yacine Chitour\\Mauricio Godoy Molina\\Petri Kokkonen}

\thanks{The work of the first author is supported by the ANR project GCM, program ``Blanche'', (project number NT09$\_$504490) and the DIGITEO-R\'egion Ile-de-France project CONGEO. The work of the second author is partially supported by the ERC Starting Grant 2009 GeCoMethods. The work of the third author is supported by Finnish Academy of Science and Letters,
KAUTE Foundation and l'Institut fran\c{c}ais de Finlande.}



\address{L2S, Universit\'e Paris-Sud XI, CNRS and Sup\'elec, Gif-sur-Yvette, 91192, France. }
\email{yacine.chitour@lss.supelec.fr}

\address{CMAP, \'Ecole Polytechnique, CNRS, 91128 Palaiseau, France.}
\email{mauricio.godoy@cmap.polytechnique.fr}

\address{L2S, Universit\'e Paris-Sud XI, CNRS and Sup\'elec, Gif-sur-Yvette, 91192, France and University of Eastern Finland, Department of Applied Physics, 70211, Kuopio, Finland.}
\email{petri.kokkonen@lss.supelec.fr}

\begin{document}

\maketitle

\begin{abstract}
In the present paper, we give a necessary and sufficient condition for a Riemannian manifold $(M,g)$ to have a reducible action of a hyperbolic analogue of the holonomy group. This condition amounts to a decomposition of $(M,g)$ as a warped product of a special form, in analogy to the classical de Rham decomposition theorem for Riemannian manifolds. As a consequence of these results and Berger's classification of holonomy groups, we obtain a simple necessary and sufficient condition for the complete controllability of the system of $(M,g)$ rolling against the hyperbolic space.
\end{abstract}


\section{Introduction}

\'E. Cartan in~\cite{cartan25} defined a geometric operation, that he called development of a manifold onto one of its tangent spaces, in order to define holonomy in terms of ``Euclidean displacements'', i.e., elements of $\Euc(n)$:
\begin{center}
\emph{``Quand on d\'eveloppe l'espace de Riemann sur l'espace euclidien tangent en $A$ le long d'un cycle partant de $A$ et y revenant, cet espace euclidien subit un d\'eplacement et tous les d\'eplacements correspondant aux diff\'erents cycles possibles forment un groupe, appel\'e groupe d'holonomie.''}
\end{center}

In the terminology of Kobayashi and Nomizu~\cite[Chapter III]{kobayashi63}, an affine (resp. linear) connection corresponds to a connection in the bundle $A(M)$ of affine frames over $M$ (resp. in the bundle $L(M)$ of frames over $M$). There is a natural one-to-one mapping between the set of affine connections on $M$ and the set of linear connections on $M$, see~\cite[Proposition 3.1, Chapter III]{kobayashi63}. Thus, the above quote is nothing but the definition of the affine holonomy group, i.e., the holonomy of the affine connection corresponding to the Levi-Civita connection, seen as a subgroup of the group of Euclidean transformations $\Euc(n)$. It is known that if $(M,g)$ is complete with an irreducible Riemannian holonomy group, the affine holonomy group contains all translations of $T|_xM$, see~\cite[Corollary 7.4, Chapter IV]{kobayashi63}. In other words, under the irreducibility hypothesis, the rotational part of the affine holonomy permits to recover the translational part, and this consists of all the possible translations in~$T|_xM$.

As regards to the geometric procedure of development introduced by Cartan, it can be generalized to the development between any two Riemannian manifolds of the same dimension. In that situation, it corresponds to the control system defined by the  rolling of one Riemannian manifold onto another one with no spin and no slip. The first definition of this generalization considered only manifolds imbedded in $\R^N$, \cite{sharpe97}. A coordinate-free definition for surfaces was introduced in \cite{agrachev99,bryant-hsu}, and later extended to general manifolds in \cite{arxiv,norway}.

Moreover, it  was observed in~\cite{CK}, that the structure of the affine holonomy group characterizes the orbits of the rolling problem, when one of the manifolds involved is the Euclidean space, and thus enables one to fully address the complete controllability issue of the rolling problem. To state this observation precisely, let us recall the definition of the rolling problem. Let $(M,g)$ and $(\hM,\hg)$ be two oriented Riemannian manifolds of dimension $n$. The configuration space of the rolling is the manifold
\[
Q=Q(M,\hM)=\{A:T|_x M\to T|_{\hat{x}} \hat{M}\ |\ A \mbox{ o-isometry},\ x\in M,\ \hat{x}\in\hat{M}\},
\]
where ``o-isometry'' stands for ``orientation preserving isometry''. An absolutely continuous curve $q(t) = (x(t),\hat x(t),A(t))$ in $Q$ is a rolling curve if $A(t)X(t)$ is parallel along $\hat x(t)$ for every vector field $X(t)$ that is parallel along $x(t)$ (no twist condition) and if $A(t)\dot x(t) = \dot{\hat x}(t)$ (no slip condition). There is a distribution $\RDist$ over $Q$, called the rolling distribution, such that the rolling curves in $Q$ are exactly the integral curves of $\RDist$.

The rolling problem is said to be completely controllable if any two points of $Q$ can be connected by a rolling curve. As is well known from control theory, a sufficient condition for the system to be completely controllable is that the evaluation of the Lie algebra generated by the vector fields in $\RDist$ at every point of the underlying manifold $Q$ is of full rank, see for instance~\cite{jurd}. As simple as this algebraic condition may seem, it turns out to be extremely hard to check for the general rolling problem. 

Moreover, it is not clear if there is in general a $G$-principal bundle structure on $Q$ making $\RDist$ a $G$-principal bundle connection. However, this is indeed the case for the projection $Q(M,\hM)\to M$, when $(\hM,\hg)$ is a space form, as shown in~\cite{CK}. More precisely, if $(\hM,\hg)$ has constant sectional curvature $c$, there is a Lie group $G_c(n)$ acting on $Q$ such that $\RDist$ is a $G_c(n)$-principal bundle connection and, moreover, its orbits are all conjugate to the holonomy group of $\RDist$, which is a subgroup of $G_c(n)$.

In the case where $c=0$, we have $G_0(n)=\Euc(n)$ and this construction reduces to study the affine holonomy group of $M$. One of the main results in~\cite{CK} shows that, provided $(M,g)$ is complete and $(\hM,\hg)$ is the Euclidean space ${\mathbb R}^n$ with the standard Riemannian structure, then the rolling system is controllable if and only if $M$ has full affine holonomy. This fact can be seen as a manifestation of De Rham's decomposition theorem since, if the holonomy of $M$ is reducible, one can detect the components of $M$ via the irreducible orbits of the distribution $\RDist$.

Up to rescaling, the cases remaining are when $c=\pm1$ and, in these cases, $G_1(n)=\SO(n+1)$ and $G_{-1}(n)=\SO_0(n,1)$, the identity component of $\SO(n,1)$. In both situations the controllability for the rolling problem can be phrased in terms of the holonomy of a connection. As shown in~\cite{CK}, there is a non-degenerate metric $h_c$ and a metric connection $\nabla^{c}$ on the vector bundle $TM\oplus\R$ over $M$, such that the rolling system is controllable if and only if the holonomy group of $\nabla^{c}$, denoted by $\mcH^{c}$, is $G_c(n)$. In the case $c=1$, the aforementioned metric is positive definite; whereas in the case $c=-1$, the metric has index one.

In the present paper, we characterize the structure of a complete and simply connected Riemannian manifold $(M,g)$ in terms of the rolling problem $Q=Q(M,\HH^n)$. Our main result states that the action of $\mcH^{-1}$ is reducible if and only if  there exists a complete simply connected Riemannian manifold $(M_1,g_1)$, so that $(M,g)$ is a warped product either of the form 
\begin{description}
\item[{\rm (WP1)}] $(\R\times M_1,\diff s^2\oplus_{e^{-s}} g_1)$, or

\item[{\rm (WP2)}] $(\HH^k\times M_1,{\bf g}_{k;-1}\oplus_{\cosh(d(\cdot))} g_1)$, 
\end{description}
where for each $x\in \HH^k$, $d(x)$ is the distance between $x$ and an arbitrary fixed point $x_0\in\HH^k$, and $1\leq k\leq n$. This can be seen as a ``hyperbolic'' analogue of the classical de Rham decomposition theorem for Riemannian manifolds,  see \cite{kobayashi63,sakai91}. By the classification theorem due to Berger in \cite{berger} (and also proved directly in \cite{dSO}), the only connected subgroup of the Lorentz group $\SO(n,1)$ that acts irreducibly on the Lorentzian space $\R^{n,1}$ is its identity component $\SO_0(n,1)$. This irreducibility criterion, together with our results, implies that the rolling problem $Q=Q(M,\HH^n)$ is not controllable if and only if $(M,g)$ decomposes into a warped product of the form {\rm (WP1)} or {\rm (WP2)}.


The case $c=1$ was addressed in \cite{CK} and it turns out to be more rigid. It is shown that if the action of $\mcH^{1}$ on the unit sphere is not transitive, then $(M, g)$ is the unit sphere with the canonical round metric. It is also important to stress that the results we present here do not correspond to the ones obtained in~\cite{wu64}. In that reference, the main result consists of an isometric decomposition of a semi-Riemannian manifold into the direct product of a semi-Riemannian irreducible submanifolds. 


The structure of the paper is the following. In Section~\ref{sec:notations} we collect results concerning the rolling problem that will be relevant in the proof of the main result. We give special emphasis to the extra symmetry that appears in the system when one of the manifolds is a space form. In Section~\ref{sec:warped}, we recall the definition of warped products, how to detect them and how to find their warping functions through a criterion due to Hiepko~\cite{hiepko}. Finally, in Section~\ref{sec:main}, we present the main results of the paper and their proofs. These results consist of a local and a global formulation of the decomposition of $(M,g)$ into a warped product of the form {\rm (WP1)} or {\rm (WP2)}, under the assumption that the action of $\mcH^{-1}$ is reducible. Both proofs are divided in two cases, depending whether the non-trivial subspace $V_1$ of $TM\oplus\R$, invariant under the action of $\mcH^{-1}$, contains a lightlike vector or not.

\section{Notations and previous results}\label{sec:notations}

Unless otherwise stated, all manifolds under consideration are smooth, connected, oriented, of finite dimension $n\geq 2$, endowed with a Riemannian metric. Similarly, all frames will be assumed positively oriented.

We intend to formulate some of the results in this paper by means of the rolling formalism presented in~\cite{arxiv,norway}. In order to do this, we need to introduce the \emph{state space $Q=Q(M,\hat{M})$} for the rolling of two $n$-dimensional connected, oriented smooth Riemannian manifolds $(M,g),(\hat{M},\hat{g})$ as
\[
Q=\{A:T|_x M\to T|_{\hat{x}} \hat{M}\ |\ A\ \textrm{o-isometry},\ x\in M,\ \hat{x}\in\hat{M}\},
\]
where ``o-isometry'' stands for ``orientation preserving isometry''.

The case in which $\hM=\R^n$ reduces to the study of the well-known concept of anti-development of curves, as observed in~\cite{GG}. The main idea consists of lifting appropriately the information about the manifold $M$ rolling on $\R^n$ to the $\GL(n)$-principal bundle of general frames. In the general case, the situation is more complicated.

\subsection{The rolling problem}

For $q=(x,\hat{x};A)\in Q$ and $X\in T|_x M$ we define the \emph{rolling lift} $\LRD(X)|_q\in T|_q Q$ as
\begin{align}\label{eq:2.5:3}
\LRD(X)|_q=\dif{t}\big|_0 (P_0^t(\hat{\gamma})\circ A\circ P_t^0(\gamma)),
\end{align}
where $\gamma,\hat{\gamma}$ are any smooth curves in $M,\hat{M}$, respectively, such that $\dot{\gamma}(0)=X$ and $\dot{\hat{\gamma}}(0)=AX$, and $P^b_a(\gamma)$ denotes the parallel transport along $\gamma$ from $\gamma(a)$ to $\gamma(b)$.

The \emph{rolling distribution} $\RDist$ on $Q$ is the $n$-dimensional smooth distribution defined, for $q=(x,\hat{x};A)\in Q$, by
\begin{align}\label{eq:2.5:1}
\RDist|_{q}=\LRD(T|_x M)|_{q}.
\end{align}

An absolutely continuous curve $t\mapsto q(t)=(x(t),\hat{x}(t);A(t))$ is a rolling curve if and only if it is almost everywhere tangent to the distribution $\RDist$, see \cite{arxiv,norway} for a description using local coordinates. We use $\mc{O}_{\RDist}(q)$ to denote the $\RDist$-orbit passing through $q$.
Similarly, if $q_0=(x_0,\hat{x}_0;A_0)\in Q$ and $\gamma:[a,b]\to M$ is a curve
such that $\gamma(a)=x_0$, we let $q_{\RDist}(\gamma,q_0)(t)$, $t\in [a,b']$,
to be the unique rolling curve through $q_0$ that projects to $\gamma$ on $M$
Here $b'\leq b$ in general; if $(M,g)$ is complete, then one can show that $b=b'$,
see \cite{CK,kobayashi63}. 





\begin{remark}
The use of the adjective ``rolling'' in the previous definitions has its origin in the classical kinematic model of one Riemannian manifold rolling over another one of the same dimension, without spinning nor slipping (cf. \cite{agrachev99,agrachev04,arxiv,norway,sharpe97}). This kinematic model can be traced back to the definition of holonomy by \'E. Cartan (cf. \cite{bryant06}) and has important applications in robotics (e.g. the plate-ball problem \cite{ACL,marigo-bicchi,murray-sastry}). The main idea in this formulation is that each point $(x,\hat{x};A)$ of the state space $Q$ can be viewed as describing a contact point of the two manifolds which is given by the points $x$ and $\hat{x}$ of $M$ and $\hat{M}$, respectively, and an isometry $A$ of the tangent spaces $T|_x M$, $T|_{\hat{x}} \hat{M}$ at this contact point, measuring the relative orientation of these tangent spaces. A curve $t\mapsto q(t)=(x(t),\hat{x}(t);A(t))$ in $Q$ is a rolling if the following constraints (see e.g. \cite{agrachev99}, \cite[Chapter 24]{agrachev04}, \cite{chelouah01}) are satisfied
\begin{itemize}
\item[(i)] The \emph{no-spinning} condition: for every absolutely continuous curve $[a,b]\to TM$; $t\mapsto X(t)$ of vectors along $t\mapsto x(t)$, we have
\be\label{eq:nospin}
\nabla_{\dot{x}(t)} X(t)=0 \quad \Longrightarrow\quad \hat{\nabla}_{\dot{\hat{x}}(t)} (A(t)X(t))=0
\quad\mbox{for almost all}\ t\in [a,b].
\ee 

\item[(ii)] The \emph{no-slipping} condition:
\be\label{eq:noslip}
A(t)\dot{x}(t)=\dot{\hat{x}}(t)\quad \mbox{for almost all}\ t\in [a,b].
\ee
\end{itemize}
\end{remark}

\subsection{Global properties of $\RDist$-orbits}

An important technical result shown in~\cite{CK} is the action of Riemannian isometries of $M$ and $\hM$ on the state space $Q$. 

\begin{proposition}\label{globalDr}
Let $F\in\Iso(M,g)$ and $\hat{F}\in\Iso(\hat{M},\hat{g})$ be Riemannian isometries of $(M,g)$ and $(\hat{M},\hat{g})$ respectively.
Define smooth free right and left actions of $\Iso(M,g)$, $\Iso(\hat{M},\hat{g})$ on $Q$ by
\[
&q_0\cdot F:=(F^{-1}(x_0),\hat{x}_0;A_0\circ F_*|_{F^{-1}(x_0)}),
\\
&\hat{F}\cdot q_0:=(x_0,\hat{F}(\hat{x}_0);\hat{F}_*|_{\hat{x}_0}\circ A_0),
\]
where $q_0=(x_0,\hat{x}_0;A_0)\in Q$.
We also set 
$
\hat{F}\cdot q_0\cdot F:=(\hat{F}\cdot q_0)\cdot F=\hat{F}\cdot (q_0\cdot F).
$
Then for any $q_0=(x_0,\hat{x}_0;A_0)\in Q$, any absolutely continuous curve $\gamma:[0,1]\to M$, $\gamma(0)=x_0$,
and any isometries $F\in\Iso(M,g)$, $\hat{F}\in\Iso(\hat{M},\hat{g})$, one has, for all $t\in [0,1]$,
\begin{align}\label{eq:equivariance_of_rolling}
\hat{F}\cdot q_{\RDist}(\gamma,q_0)(t)\cdot F
=q_{\RDist}(F^{-1}\circ\gamma,\hat{F}\cdot q_0\cdot F)(t).
\end{align}
In particular,
$
\hat{F}\cdot\mc{O}_{\RDist}(q_0)\cdot F=\mc{O}_{\RDist}(\hat{F}\cdot q_0\cdot F).
$
\end{proposition}

\vspace{0.3cm}

An initial reduction of the problem is the fact that the controllability question for the rolling problem for $M$ and $\hM$ is equivalent to study the controllability of Riemannian coverings of $M$ and $\hM$ rolling against each other (cf.~\cite{arxiv}). An immediate consequence, is that one can assume with no loss of generality that both manifolds $M$ and $\hM$ are simply connected.

\subsection{Space forms and their isometry groups}


The $n$-dimensional space form $\FF^n_{c}$ of curvature $c\neq0$ as a subset of $\R^{n+1}$, $n\geq 1$, 
given by
\[
\FF^n_{c}:=\big\{(x_1,\dots,x_{n+1})\in\R^{n+1}\ |&\ c(x_1^2+\dots+x_n^2)+x_{n+1}^2=1, \\
&x_{n+1}+\frac{c}{|c|}\geq 0\big\}.
\]
Equip $\FF^n_{c}$ with a Riemannian metric ${\bf g}_{n;c}$ defined as 
the restriction to $\FF^n_{c}$ of the non-degenerate symmetric $(0,2)$-tensor 
$
s_{n;c}:=(\diff x_1)^2+\dots+(\diff x_n)^2+c^{-1}(\diff x_{n+1})^2.
$
The condition $x_{n+1}+\frac{c}{|c|}\geq 0$ in the definition of $\FF^n_{c}$
guarantees that $\FF^n_{c}$ is connected also when $c<0$.
We denote, as usual, $\FF^n_{1}$ and $\FF^n_{-1}$ by $S^n$ and $\HH^n$ respectively.

\begin{remark}
Note that in the definition above there is an underlying continuity with respect to the curvature parameter $c$, once we disregard the connectedness assumption. More precisely, the set
\[
\big\{(x_1,\dots,x_{n+1})\in\R^{n+1}\ |&\ c(x_1^2+\dots+x_n^2)+x_{n+1}^2=1\big\}
\]
consists of the two hyperplanes $x_{n+1}=\pm 1$ when $c=0$, a two-sheeted hyperboloid with fixed vertices $(0,\ldots,0,\pm1)$ and foci $(0,\ldots,0,\pm\frac{c-1}{c})$ when $c<0$, and an ellipse with vertices $(0,\ldots,0,\pm1)$ in the $x_{n+1}$-axis and semiaxes of length $\frac{1}{\sqrt{c}}$ on the hyperplane $x_{n+1}=0$. This is in accordance with the definition of the tensor $s_{n;c}$ since for it to behave well when $c\to0^+$ or $c\to0^-$, one needs to impose $\diff x_{n+1}=0$ when $c=0$.
 \end{remark}
 
 \vspace{0.3cm}

Let $G_c(n)$ be the identity component of the Lie group of linear maps $\R^{n+1}\to \R^{n+1}$
that leave invariant the bilinear form
\[
\la x,y\ra_{n;c}:=\sum_{i=1}^n x_i y_i+c^{-1}x_{n+1}y_{n+1},
\]
for $x=(x_1,\dots,x_{n+1})$, $y=(y_1,\dots,y_{n+1})$. Observe that $G_1(n)=\SO(n+1)$ and $G_{-1}(n)=\SO_0(n,1)$, the identity component of $\SO(n,1)$.

If $c=0$, the space form $(\FF^n_0,{\bf g}_{n;0})$ is simply equal to $\R^n$ with the Euclidean
metric, $G_{0}(n)$ is set to be the group $\Euc(n)$, the special Euclidean group of $(\FF^n_0,{\bf g}_{n;0})$.
Recall that $\mathrm{SE}(n)$ 
is equal to $\R^n\times \SO(n)$ as a set,
and is equipped with the group operation $\star$ given by
\[
(v,L)\star (u,K):=(Lu+v,L\circ K).
\]
The natural action, also written as $\star$, of $\SO(n)$ on $\R^n$
is given by
\[
(u,K)\star v:=Kv+u,\quad (u,K)\in\SO(n),\ v\in \R^n.
\]
Finally recall that, with this notation,
the isometry group of $(\FF^n_c,{\bf g}_{n;c})$ is equal to $G_c(n)$ 
for all $c\in\R$ (cf. \cite{kobayashi63}).

\subsection{Reduction of the rolling problem}\label{sec:reduction}

When rolling against a space form, it is possible to reduce the controllability problem to the study of certain holonomy groups. In other words, one can consider the change of the initial state of the system after rolling along piecewise $C^1$-loops in $M$ based at $x$.



The fundamental feature of rolling over a space form lies in the fact that there is a $G_c(n)$-principal bundle structure for the state space compatible with the distribution $\RDist$, i.e. $\RDist$ is a $G_c(n)$-principal bundle connection. This result was proved in~\cite{CK} by using Proposition~\ref{globalDr}, and it is provided below.
\begin{proposition}\label{pr:principal}
Let $Q=Q(M,\FF^n_c)$ be the configuration space of rolling $M$ against the space form $\FF^n_c$. The following hold
\begin{itemize}
\item[(i)]
The bundle $\pi_{Q,M}\colon Q\to M$ is a principal $G_c(n)$-bundle
with a left action $\mu:G_c(n)\times Q\to Q$ defined for every $q=(x,\hat{x};A)$ by
\[
\mu((\hat{y},C),q)=&(x,C\hat{x}+\hat{y};C\circ A), && \textrm{if $c=0$}, & \\
\mu(B,q)=&(x,B\hat{x};B\circ A), && \textrm{if $c\neq 0$}. &
\]
Moreover, the action $\mu$ preserves the distribution $\RDist$
i.e., for any $q\in Q$ and $B\in G_c(n)$,
$
(\mu_B)_*\RDist|_q=\RDist|_{\mu(B,q)}
$
where $\mu_B:Q\to Q$; $q\mapsto \mu(B,q)$.

\item[(ii)]
For any given $q=(x,\hat{x};A)\in Q$,
there is a unique subgroup $\mc{H}_q$ of $G_c(n)$,
called the holonomy group of $\RDist$,
such that
\[
\mu(\mc{H}_q\times\{q\})=\mc{O}_{\RDist}(q)\cap \pi_{Q,M}^{-1}(x).
\]
In addition, if $q'=(x,\hat{x}';A')\in Q$ is in the same $\pi_{Q,M}$-fiber as $q$,
then $\mc{H}_q$ and $\mc{H}_{q'}$ are conjugate in $G_c(n)$
and all conjugacy classes of $\mc{H}_q$ in $G_c(n)$ are of the form~$\mc{H}_{q'}$.
\end{itemize}
\end{proposition}

\vspace{0.3cm}

An open problem related to the proposition above asks for the extent to which this result holds. More precisely, given two Riemannian manifolds $M$ and $\hM$ of dimension $n\geq3$ and the canonical projection $\pi_{Q,M}\colon Q=Q(M,\hM)\to M$, can one give conditions on the manifolds so that there exists a $G$-principal bundle structure for some Lie group $G$ so that the rolling distribution $\RDist$ is $G$-equivariant? For instance, this is true if one of the manifolds is a space form.

For the case $c=0$, one can take advantage of the semi-direct product structure of $\Euc(n)$ by considering the projection of the orbit onto $\SO(n)$, which is nothing but the Riemannian holonomy group of $M$. As a result, it is proved in~\cite{CK} that complete controllability holds if and only if $M$ has full holonomy.

For the case when $c\neq0$, the problem is more subtle. It was shown in~\cite{CK} that this principal $G_c(n)$-bundle structure implies the existence of a vector bundle connection $\nabla^{c}$ on the vector bundle $\pi_{TM\oplus\R}:TM\oplus\R\to M$, called the \emph{rolling connection}, defined as follows: for every $x\in M$, $X\in T|_x M$, $(Y,s)\in\VF(M)\times\Cinf(M)$,
\begin{align}\label{eq:nabla_rol_explicit}
\nabla^{c}_X (Y,s)=\Big(\nabla_X Y+s(x)X,X(s)-cg\big(Y|_x,X)\Big).
\end{align}
Here we have canonically identified the space of smooth sections $\Gamma(\pi_{TM\oplus\R})$
of $\pi_{TM\oplus\R}$ with $\VF(M)\times \Cinf(M)$.

The connection $\nabla^{c}$ is a metric connection with respect to the fiber inner product $h_c$ on $TM\oplus\R$ defined by
\[
h_c((X,r),(Y,s))=g(X,Y)+c^{-1}rs,
\]
where $X,Y\in T|_x M$, $r,s\in\R$. Its holonomy group is denoted by $\mcH^{c}$.

After a trivial scaling, it is enough to consider only the cases $c=\pm1$. The use of the rolling connection $\nabla^{c}$ on the vector bundle $TM\oplus\R$ has the advantage that it allows one to prove that complete controllability of the rolling system is equivalent to the fact that $\mcH^{c}$ equals $\SO(n+1)$ for the spherical case $c=1$ and $\SO_0(n,1)$ for the hyperbolic case $c=-1$.




\section{Warped products}\label{sec:warped}

In order to present our results, we need some standard material on warped products, as presented for example in~\cite{oneill83}, as well as means to detect when a manifold can be decomposed as the warping of two (or more) manifolds.

\subsection{Definitions}

\begin{definition}\label{def:warped}
\begin{itemize}
\item[(i)] Let $(N,h),(M,g)$ be Riemannian manifolds
and $f\in\Cinf(N)$ a non-vanishing function.
Then the manifold $N\times M$ equipped with the metric
\[
(h\oplus_{f} g)|_{(y,x)}:=h|_y+f(y)^2 g|_x,\quad (y,x)\in N\times M,
\]
is a Riemannian manifold called the \emph{warped product} of $(N,h)$
and $(M,g)$ with warping function $f$.

\item[(ii)] Let $(N,h),(M_1,g_1),(M_2,g_2)$ be Riemannian manifolds
and $f_1,f_2\in\Cinf(N)$. Denote by $\pr_1:N\times M_1\to N$.
Then $(N\times M_1\times M_2,(h\oplus_{f_1} g_1)\oplus_{\pr_1^*(f_2)} g_2)$
is called the \emph{doubly warped product} of $(N,g),(M_1,g_1),(M_2,g_2)$
with warping functions $f_1,f_2$.
We denote its metric simply by $h\oplus_{f_1}g_1\oplus_{f_2} g_2$.
\end{itemize}
\end{definition}

\begin{remark}
Note that the metric of the above doubly warped product at $(y,x_1,x_2)\in N\times M_1\times M_2$ has the form
\[
(h\oplus_{f_1}g_1\oplus_{f_2} g_2)|_{(y,x_1,x_2)}=h|_{y}+f_1(y)^2 g_1|_{x_1}+f_2(y)^2 g_2|_{x_2}.
\]
Therefore, it is easy to see that $(N\times M_1\times M_2,h\oplus_{f_1}g_1\oplus_{f_2} g_2)$ and $(N\times M_2\times M_1,h\oplus_{f_2}g_2\oplus_{f_1} g_1)$ are isometric.
\end{remark}


\subsection{Detecting warped products}

\begin{definition}\label{def:spherical}
A submanifold $N$ of a Riemannian manifold $(M,g)$ is \emph{spherical} if there is a local section $\nu$ of the normal bundle $TN^\perp$ such that:
\begin{itemize}
\item The second fundamental form $\II_N$ of $N$ has the form 
\[
\II_N(X,Y)=g(X,Y)\nu,\quad \forall X,Y\in T|_x N,\ x\in N.
\]

\item The section $\nu$ satisfies
\begin{align}\label{sphericalmfld}
\nabla_X \nu\in TN,\quad \forall X\in TN.
\end{align}
\end{itemize}
\end{definition}

The last condition \eqref{sphericalmfld} means that $\nu$ is parallel with respect to the \emph{normal} connection of $N$.

\begin{theorem}\label{th:warped} (\cite{hiepko})
Let $(M,g)$ be a Riemannian manifold
and suppose there is a smooth constant rank distribution $\mc{D}$ on $M$
with the following properties:
\begin{itemize}
\item[(i)] Both $\mc{D}$ and $\mc{D}^\perp$ are integrable.
\item[(ii)] The integral manifolds of $\mc{D}^\perp$ are totally geodesic.
\item[(iii)] The integral manifolds of $\mc{D}$ are spherical.
\end{itemize}
Then $(M,g)$ is locally a warped product.
If moreover $(M,g)$ is complete and simply connected,
then $(M,g)$ is globally a warped product.
Finally, if $f$ is the warping function and $\nu$ is a section of the bundle $\mc{D}^\bot$ as in Definition \ref{def:warped}, then
\[
\nu=-\frac{\nabla f}{f}.
\]
\end{theorem}

\begin{remark}\label{re:warped}
More precisely, as explained in \cite{hiepko} (see Eqs. (11) and (17) there),
under the assumptions of the above theorem,
every $x\in M$ has a neighbourhood $U$
and integral manifolds $N,N^\perp$ through $x$ of $\mc{D}$, $\mc{D}^\perp$, respectively, 
such that $U$ is diffeomorphic to $N^\perp\times N$
which maps $g|_U$ to $g|_{N^\perp}\oplus_f h$, where $h$ is a certain metric on $N$.
If $(M,g)$ is complete and simply connected, one may take $U=M$.
\end{remark}

\section{Presentation of the main results}\label{sec:main}

We now present the main global result of the present paper.

\begin{theorem}\label{th:main}
Let $(M,g)$ be a complete and simply connected Riemannian manifold.
For $c<0$, the rolling holonomy group $\mcH^c$ is reducible,
if and only if $(M,g)$ is a warped product either of the form
\begin{description}
\item[{\rm (WP1)}] $(\R\times M_1,\diff s^2\oplus_{e^{cs}} g_1)$, or

\item[{\rm (WP2)}] $( \FF^k_c\times M_1,{\bf g}_{k;c}\oplus_{\cosh(\sqrt{-c}\,d(\cdot))} g_1)$, where $1\leq k\leq n$ and for each $x\in \FF^k_c$ , $d(x)$ is the distance between $x$ and an arbitrary fixed point $x_0\in\FF^k_c$,
\end{description}
where $(M_1,g_1)$ is a complete simply connected Riemannian manifolds of lower dimension.
\end{theorem}


From the previous result one immediately deduces the following
characterization of complete controllability of the rolling problem against the hyperbolic space $\HH^n$.

\begin{corollary}
Let $(M,g)$ be a complete, oriented and simply connected Riemannian $n$-manifold rolling against the space form $(\HH^n,{\bf g}_{n;-1})$ of curvature $-1$.
Then the associated rolling problem is completely controllable if and only if $(M,g)$ is not isometric to a warped product of the form {\rm (WP1)} or {\rm (WP2)}.
\end{corollary}

\begin{proof}
With the notations of Theorem \ref{th:main} and Subsection \ref{sec:reduction}, studying the rolling problem reduces to determining the holonomy group $\mcH^{-1}$.  Assume that $(M,g)$ is of the form {\rm (WP1)} or {\rm (WP2)}, then $\mcH^{-1}$ is a proper subgroup of $\SO_0(n,1)$, i.e., the rolling problem is not controllable according to \cite{CK}. On the other hand, if $(M,g)$ is not of the form {\rm (WP1)} or {\rm (WP2)}, then the action of $\mcH^{-1}$ must be irreducible. Since $M$ is simply connected, then $\mcH^{-1}$ is connected, and thus it is a connected subgroup of $\SO(n,1)$. Therefore it equals $\SO_0(n,1)$, according to \cite{berger,dSO}.
\end{proof}



\begin{remark}\label{rem:deRhamc=0}
Note that the statement of Theorem~\ref{th:main} extends to the case $c=0$, by setting the warping function equal to 1. Moreover, using an extra induction argument, one immediately recovers the de Rham decomposition theorem.
\end{remark}




\subsection{Proof of the main result}

The study of reducibility of $\mc{H}^c$
in the case $c=0$ corresponds to the classical de Rham theorem, as mentioned in Remark~\ref{rem:deRhamc=0}, and for $c=1$ this was done in \cite{CK}. The rest of the paper is devoted to the proof of Theorem~\ref{th:main}. By rescaling, we may assume without loss of generality, that $c=-1$.
Theorem~\ref{th:main} is a consequence of the following two propositions.

\begin{proposition}\label{pr:main-local}
With the notation above, assume that the holonomy group $\mcH^{-1}$ is reducible.
Then $M$ is locally of one of the following forms:
\begin{itemize}
\item[{\rm (LW1)}] a warped product $(I\times M_1,\diff s^2\oplus_{e^{-s}} g_1)$,
where $I\subset\R$ is an interval;
\item[{\rm (LW2)}] a doubly warped product $(I\times M_1\times M_2,\diff s^2\oplus_{\sinh(s)}g_1 \oplus_{\cosh(s)} g_2)$;
\item[{\rm (LW3)}] a warped product $(O\times M_1, {\bf g}_{k;-1}\oplus_{\cosh(d(\cdot))} g_1)$,
where $O\subset \HH^k$ is a normal neighbourhood of a point $x_0\in O$
 and $d$ is a distance function from $x_0$ in $\HH^k$,
\end{itemize}
where $(M_1,g_1)$, $(M_2,g_2)$ are Riemannian manifolds of lower dimension. 
\end{proposition}


\begin{proposition}\label{pr:converse}
Suppose $(M,g)$ is a doubly warped product
of one of the above forms {\rm (LW1)}, {\rm (LW2)} or {\rm (LW3)}.
Then the holonomy group $\mcH^{-1}$ is reducible.

\end{proposition}

\begin{remark}
In the previous propositions, it is possible to replace  $(I\times M_1,\diff s^2\oplus_{e^{-s}} g_1)$ by $(-I\times M_1,\diff s^2\oplus_{e^s} g_1)$, since the map $(s,x_1)\mapsto(-s,x_1)$ provides an isometry between them.
\end{remark}

Note that both propositions are of local nature. Along the respective arguments, we will provide the necessary modifications to derive the full proof of Theorem \ref{th:main}

Before starting with the proofs, we need to introduce some more notations. The metric $h:=h_{-1}$ associated to the bundle $\pi_{TM\oplus\R}:TM\oplus\R\to M$ is then
\[
h((X,r),(Y,s))=g(X,Y)-rs,\quad (X,r),(Y,s)\in T|_x M\oplus\R.
\]
Moreover, the linear connection $\nabla^{-1}$ is given by
\[
\nabla^{-1}_X (Y,s)=(\nabla_X Y+sX,X(s)+g(X,Y)).
\]
for every $X,Y\in\VF(M),\ s\in\Cinf(M)$. In particular, if $\gamma$ is a unit speed geodesic on $M$
and $(Y(t),s(t))$ is parallel along $\gamma$, then
\[
\begin{cases}
\nabla_{\dot{\gamma}} Y+s\dot{\gamma}=0, \\
\dot{s}+g(\dot{\gamma},Y)=0.
\end{cases}
\]
Differentiating once more and simplifying we get
\[
\begin{cases}
\nabla_{\dot{\gamma}}\nabla_{\dot{\gamma}} Y=g(\dot{\gamma},Y)\dot{\gamma}, \\
\ddot{s}-s=0.
\end{cases}
\]

\subsection{Proof of Proposition \ref{pr:main-local}}

In this section, we provide the proof of Proposition \ref{pr:main-local}
and the proof of the condition of necessity in Theorem \ref{th:main}.
The sufficiency for Theorem \ref{th:main} is proved in the next section.

Let $(V,h)$ be a Lorentzian vector space. For a vector subspace $W\subset V$, we define
\[
W^{\perp_h}=\{v\in V\ |\ h(v,w)=0,\forall w\in W\},
\]
the $h$-orthogonal space to $W$.
We will occasionally use a notation $\n{v}_h^2:=h(v,v)$, when $v\in V$.

 Let $V_1$ be a vector subbundle of $TM\oplus\R$ invariant under the holonomy
group $\mcH^{-1}$ of $\nabla^{-1}$
and set $V_2=V_1^{\perp_h}$.
This is again invariant under $\mcH^{-1}$, since $\nabla^{-1}$ is metric with respect to $h$.
Since $\dim (V_1\cap V_2)\in \{0,1\}$, the argument is divided into two cases.

\subsubsection{Case $V_1\cap V_2=\{0\}$} 

We have $TM\oplus\R=V_1\oplus V_2$.
For $\alpha=1,2$, define the subsets $N_\alpha$ of $M$ by
\[
N_\alpha=\{x\in M\ |\ (0,1)\in V_\alpha|_x\}.
\]
The restrictions of $h$ to $V_1$ and $V_2$
are both non-degenerate, and since $h$ has signature $(n,1)$, $h$ is positive definite on one of them,
which we assume without loss of generality to be $V_2$. Let us assume $h|_{V_2}$ has signature $(n-m,0)$, for some $m$ such that $0\leq m<n$.
Therefore $h$ is Lorentzian on $V_1$, i.e. $h|_{V_1}$ has signature $(m,1)$.
In particular, $V_1$ intersects transversally the light cone.
To this end, notice that since $\nabla^{-1}$ is a metric connection,
it preserves the signatures of invariant subbundles $V_1,V_2$ so the above claims are well established.

First we prove that $N_2$ is empty and $N_1$ is non-empty 
in the case where $M$ is complete.

\begin{lemma}\label{le:N1-N2}
One has $N_2=\emptyset$
and if $M$ is complete, then $N_1\neq \emptyset$.
\end{lemma}

\begin{proof}
The fact that $N_2=\emptyset$ is trivial 
because 
\[h((0,1),(0,1))=\|(0,1)\|_h^2=-1,\]
and $h$ is positive definite on $V_2$.

Suppose that $M$ is complete and fix $x_0\in M$.
Since $h$ is Lorentzian on $V_1$,
there is a $(X_0,r_0)\in V_1|_{x_0}$ such that $\n{(X_0,r_0)}^2_h<0$.
By scaling, we can assume that $\n{(X_0,r_0)}^2_h=-1$,
i.e., $\n{X_0}^2_g-r_0^2=-1$ and $r_0>0$.
If $X_0=0$, then $r_0= 1$ and $ (0,1)\in V_1|_{x_0}$ and we are done.
Hence assume that $X_0\neq 0$.
Let $\gamma$ be a unit speed geodesic with velocity $X_0/\n{X_0}_g$ and write $(X(t),r(t))$
for the $\nabla^{-1}$-parallel transport of $(X_0,r_0)$ along $\gamma$.
Since $r(0)=r_0$ and $\dot{r}(0)=-g(\dot{\gamma}(0),X_0)=-\n{X_0}_g$,
and because $\ddot{r}-r=0$, we get
\[
r(t)=r_0\cosh(t)-\n{X_0}_g\sinh(t).
\]
Since  $\n{X_0}^2_g-r_0^2=-1$ and $r_0>0$, there exists a unique $t_1\in\R$
such that $(\cosh(t_1),\sinh(t_1))=(r_0,\n{X_0}_g)$.
Hence $r(t_1)=r_0^2-\n{X_0}_g^2=1$.
But then
\[
\n{X(t_1)}_g^2-1=&\n{X(t_1)}_g^2-r(t_1)^2=\n{(X(t_1),r(t_1))}^2_h\\
=&\n{(X_0,r_0)}^2_h=-1,
\]
which implies that $\n{X(t_1)}_g^2=0$
and hence $(0,1)=(X(t_1),r(t_1))\in V_1|_{\gamma(t_1)}$
i.e. $\gamma(t_1)\in N_1$. This finishes the proof.
\end{proof}

For $\alpha=1,2$, let $\pi_{V_\alpha}:=\pi_{TM\oplus\R}|_{V_\alpha}:V_\alpha\to M$ and
define smooth sections $(W_\alpha,w_\alpha)\in \Gamma(\pi_{V_\alpha})$,
such that at every point $x\in M$,
\[
(0,1)=(W_1,w_1)+(W_2,w_2).
\]
Clearly then $W_1=-W_2,\quad w_1+w_2=1$.

The fact $N_2=\emptyset$ means that $w_1$ never vanishes on $M$.
Indeed, if $w_1=0$ at some point, then $w_2=1$
and
\[
-1=\n{(0,1)}_h^2=\n{(W_1,0)}_h^2+\n{(W_2,1)}_h^2
=\n{W_1}_g^2+\n{W_2}_g^2-1,
\]
hence $W_1=0$, $W_2=0$ and $V_2\ni (W_2,w_2)=(0,1)$, a contradiction.

A simple calculation shows that the curvature $R^{\nabla^{-1}}$ of the rolling connection $\nabla^{-1}$ is given by
\[
R^{\nabla^{-1}}((X,r),(Y,s))(Z,u)=(R(X,Y)Z+B(X,Y)Z,0),
\]
where $B(X,Y)Z:=g(Y,Z)X-g(X,Z)Y$.

\begin{lemma}\label{le:curvature}
For all $x\in M$ and $X,Y\in T|_x M$, one has
\[
R(X,Y)W_\alpha=-B(X,Y)W_\alpha,\quad \alpha=1,2.
\]
\end{lemma}

\begin{proof}
Notice that for any $(X,r),(Y,s)\in T|_{(x,t)}(M\times\R)$ one has 
\[
R^{\nabla^{-1}}((X,r),(Y,s))(0,1)=(R(X,Y)0+B(X,Y)0,0)=(0,0).
\]
On the other hand, if ${\mathfrak{h}}^{-1}|_x$ denotes the Lie algebra of $\mcH^{-1}|_x$,
by the Ambrose-Singer theorem
$R^{\nabla^{-1}}((X,r),(Y,s))\in {\mathfrak{h}}^{-1}|_{x}$, so
\[
R^{\nabla^{-1}}((X,r),(Y,s))V_\alpha|_{x}\subset V_\alpha|_{x},\quad \alpha=1,2.
\]
Hence
\[
(0,0)=&R^{\nabla^{-1}}((X,r),(Y,s))(0,1) \\
=&\underbrace{R^{\nabla^{-1}}((X,r),(Y,s))(W_1,w_1)}_{\in V_1}+\underbrace{R^{\nabla^{-1}}((X,r),(Y,s))(W_2,w_2)}_{\in V_2},
\]
because $(W_\alpha,w_\alpha)\in V_\alpha$, $\alpha=1,2$.
Therefore, since $V_1\cap V_2=\{0\}$, we have
\[
R^{\nabla^{-1}}((X,r),(Y,s))(W_\alpha,w_\alpha)=(0,0),\quad \alpha=1,2,
\]
which means that
\[
R(X,Y)W_\alpha+B(X,Y)W_\alpha=0,\quad \alpha=1,2,
\]
and hence the claim has been established.
\end{proof}

Define for every $x\in M$,
\[
V_\alpha^M|_x:=\{X\ |\ (X,r)\in V_\alpha\}\subset T|_x M,\quad \alpha=1,2.
\]
Clearly $V_\alpha^M$ is a smooth distribution on $M\setminus N_\alpha$
with $\rank V_\alpha^M=\rank V_\alpha$.
In particular, $V_2^M$ is a smooth constant rank distribution on all of $M$, since $N_2=\emptyset$.
Moreover, it is clear that $V_1^M$ is a smooth non-constant rank distribution
so that $\rank V_1^M=\rank V_1-1=m$ at points $x\in N_1$.

\begin{lemma}\label{le:V1McapV2M}
For every $x\in M$,
the intersection $V_1^M\cap V_2^M$ is spanned
by $W_1$ ($=-W_2$)
and so is one dimensional on $M\setminus N_1$ and zero on $N_1$.
\end{lemma}

\begin{proof}
Indeed, if $X\in V_1^M\cap V_2^M$,
then there are $r_1,r_2\in\R$ such that
$(X,r_\alpha)\in V_\alpha$, $\alpha=1,2$. But then
one has
\[
\underbrace{(X,r_1)}_{\in V_1}-\underbrace{(X,r_2)}_{\in V_2}&=(0,r_1-r_2)=(r_1-r_2)(0,1)\\
&=(r_1-r_2)\underbrace{(W_1,w_1)}_{\in V_1}+(r_1-r_2)\underbrace{(W_2,w_2)}_{\in V_2},
\]
and since $V_1\cap V_2=\{0\}$,
one has
\[
(r_1-r_2)(W_1,w_1)=(X,r_1), \\
(r_2-r_1)(W_2,w_2)=(X,r_2).
\]
In particular, $X=(r_1-r_2)W_1$,
which shows that $V_1^M\cap V_2^M\subset \R W_1$.
Finally, since $W_1\in V_1^M$, $W_2\in V_2^M$ and $W_1=-W_2$,
we have that $\R W_1\subset V_1^M\cap V_2^M$.
\end{proof}

Define $\mc{D}_1:=(V_2^M)^{\perp}$ and $\mc{D}_2:=(V_1^M)^\perp$
i.e. the orthogonal complements of $V_2$ and $V_1$ with respect to $g$.
Notice that $\mc{D}_\alpha\subset V_\alpha$ for $\alpha=1,2$.
Since $V_2^M$ is a smooth constant rank distribution on $M$
then so is $\mc{D}_1$ as well and $\rank \mc{D}_1=m$.
Similarly, $\mc{D}_2$ has constant rank $n-m-1$ on $M\setminus N_1$
and rank $n-m$ on $N_1$.
It is obvious that $\mc{D}_2$ is a smooth distribution on $M\setminus N_1$.
However, it is not continuous at points of $x\in N_1$.
Indeed, as will be proved in Lemma \ref{le:intDi} below,
$N_1$ is a submanifold of $M$ positive codimension
(indeed it has dimension $m$).
But the rank of a continuous distribution is lower semicontinuous
and hence can only locally increase,
while $\mc{D}_2$ has rank $n-m$ on the nowhere dense set $N_1$
which is higher than its rank $n-m-1$ on $M\setminus N_1$,
so $\mc{D}_2$ cannot be continuous at points of $N_1$.

\begin{lemma}\label{le:intDi}
Let $\{\alpha,\beta\}=\{1,2\}$.
The distribution $\mc{D}_{\alpha}$ is integrable on $M\setminus N_{\beta}$
and the set $N_1$ is an integral manifold of $\mc{D}_1$ which is embedded in $M$.
Moreover, if $O$ is an integral manifold of $\mc{D}_{\alpha}$ 
and if $\II_O$ is its second fundamental form, then
for every $X,Y\in T|_x O$, $x\in M\setminus N_{\beta}$,
\[
\II_O(X,Y)=\frac{g(X,Y)}{w_\alpha}W_{\alpha}.
\]
In particular, each integral manifold of $\mc{D}_{\alpha}$ is umbilical
and $N_1$ is totally geodesic
\end{lemma}

\begin{proof}
Recall that $\mc{D}_\alpha=(V_{\beta}^M)^\perp$
and $TM\oplus\R=V_\alpha\oplus V_{\beta}$.
Suppose $X,Y$ are vector fields tangent to $\mc{D}_\alpha$ on $M\setminus N_{\beta}$.
Then
\[
\{0\}=g(\{Y\}\times V_{\beta}^M)=h(\{(Y,0)\}\times V_{\beta}),
\]
so $(Y,0)\in V_\alpha$, and similarly $(X,0)\in V_\alpha$. 
Hence
\[
V_\alpha\ni \nabla^{-1}_X (Y,0)=(\nabla_X Y,g(X,Y)).
\]
Similarly, $(\nabla_Y X,g(Y,X))\in V_\alpha$ and thus
\[
([X,Y],0)=(\nabla_X Y,g(X,Y))-(\nabla_Y X,g(Y,X))\in V_\alpha.
\]
Therefore
\[
g(\{[X,Y]\}\times V_{\beta}^M)=h(\{([X,Y],0)\}\times V_{\beta})=0,
\]
so $[X,Y]$ is tangent to $(V_{\beta}^M)^{\perp}=\mc{D}_\alpha$.
This proves that $\mc{D}_\alpha$ is involutive and hence integrable on $M\setminus N_{\beta}$.

Let $O$ be an integral manifold of $\mc{D}_{\alpha}$ in $M\setminus N_{\beta}$
and let $X,Y$ be tangent to $O$.
By what we have shown above,
\[
V_\alpha\ni \nabla^{-1}_X (Y,0)=&(\nabla_X Y,g(X,Y))
=(\nabla_X Y,0)+g(X,Y)(0,1) \\
=&(\nabla_X Y,0)+g(X,Y)(W_\alpha,w_\alpha)+g(X,Y)(W_{\beta},w_{\beta}) \\
=&(\nabla_X Y+g(X,Y)W_{\beta},g(X,Y)w_{\beta})+g(X,Y)(W_\alpha,w_\alpha),
\]
and so
\[
(\nabla_X Y+g(X,Y)W_{\beta},g(X,Y)w_{\beta})\in V_\alpha.
\]
Since also $w_{\beta}(\nabla_X Y,g(X,Y))\in V_\alpha$, it follows that
\[
((1-w_{\beta})\nabla_X Y+g(X,Y)W_{\beta},0)\in V_\alpha,
\]
and hence, since $1-w_{\beta}=w_\alpha$ and $W_{\beta}=-W_\alpha$,
\[
0=&h(\{(w_\alpha\nabla_X Y-g(X,Y)W_\alpha,0)\}\times V_{\beta})\\
=&g(\{w_\alpha\nabla_X Y-g(X,Y)W_\alpha\}\times V_{\beta}^M).
\]
Thus $w_\alpha\nabla_X Y-g(X,Y)W_\alpha\in \mc{D}_\alpha$.
Since $W_\alpha=-W_{\beta}\in V_{\beta}^M=\mc{D}_\alpha^\perp$, this 
proves that
\[
\II_O(X,Y)=\frac{g(X,Y)}{w_\alpha}W_\alpha.
\]

We show that $N_1$ is an integral manifold of $\mc{D}_1$.
Indeed, let $x_1\in N_1$ and let $(Y_i,s_i)$, $i=1,\dots,n-m$,
be a local basis of $V_2$ on an open set $U\ni x_1$.
Since $h$ is positive definite on $V_2$,
we may assume that the basis $(Y_i,s_i)$, $i=1,\dots,n-m$
is $h$-orthonormal.
Moreover, if $x\in N_1$, then for all $i$, $s_i(x)=-h((0,1),(Y_i|_x,s_i(x)))=0$
since $(0,1)\in V_1|_x$.

Define $F:U\to\R^{n-m}$ by
\[
F=\big(h((Y_1,s_1),(0,1)),\dots,h((Y_{n-m},s_{n-m}),(0,1))\big),
\]
and notice that $F^{-1}(0)=N_1\cap U$. To show that $N_1$
is a smooth embedded submanifold of dimension $m$, it thus suffices to show that
$F$ is a submersion at every point $x\in N_1\cap U$.
But if $x\in N_1\cap U$ and $k=1,\dots,n-m$, then
\[
F_*|_x(Y_k)=&\big(h(\nabla^{-1}_{Y_k|_x}(Y_i,s_i),(0,1))+h((Y_i|_x,s_k(x)),\nabla^{-1}_{Y_k}(0,1))\big)_{i=1}^{n-m},
\]
because $\nabla^{-1}$ is metric with respect to $h$.
Since $\nabla^{-1}_{Y_k|_x}(Y_i,s_i)\in V_2|_x$, and $(0,1)\in V_1|_x$,
the term $h(\nabla^{-1}_{Y_k|_x}(Y_i,s_i),(0,1))$ vanishes.
Moreover
\[
h((Y_i|_x,s_k(x)),\nabla^{-1}_{Y_k}(0,1))&=h((Y_i|_x,s_i(x)),(Y_k|_x,0))\\
&=g(Y_i|_x,Y_k|_x)=\delta_{ik},
\]
since $s_i(x)=s_k(x)=0$.
Hence if $e_i$, $i=1,\dots,n-m$, is the canonical basis of $\R^{n-m}$,
then for all $x\in N_1\cap U$, $F_*|_x(Y_k)=e_k$, $k=1,\dots,n-m$
and so they are linearly independent. Hence $F$ is a submersion at every point of $N_1\cap U$.

To show that $T|_x N_1=\mc{D}_1|_x$ for all $x\in U\cap N_1$,
notice that if $X\in \mc{D}_1|_x$, then by computation as above,
\[
F_*|_x(X)=\big(h(\nabla^{-1}_{X}(Y_i,s_i),(0,1))+g(Y_i|_x,X)\big)_{i=1}^{n-m}=0,
\]
because $\nabla^{-1}_{X}(Y_i,s_i)\in V_2|_x$, $(0,1)\in V_1|_x$
and $Y_i|_x\in V_2^M|_x$ while $X\in \mc{D}_1|_x=(V_2^M|_x)^\perp$.
This shows that $\mc{D}_1|_x\subset T|_x N_1$ for all $x\in N_1\cap U$
and since both linear spaces have dimension $m$,
we have the equality i.e. $N_1$ is an integral manifold of $\mc{D}_1$.


Finally, since $N_1$ is an integral manifold of $\mc{D}_1$
and since $(W_1,w_1)=(0,1)$ on $N_1$,
one has that the second fundamental form $\II_{N_1}$
vanishes on $N_1$. Therefore $N_1$ is totally geodesic.
\end{proof}

In particular,  at every $x\in N_1$ one has $V_1^M|_x=\mc{D}_1|_x=T|_x N_1$.

\begin{lemma}\label{le:spherical}
Let $\{\alpha,\beta\}=\{1,2\}$.
The integral manifolds of $\mc{D}_\alpha$ in $M\setminus N_\beta$ are spherical.
\end{lemma}

\begin{proof}
We need to show that $\nabla_X (w_\alpha^{-1}W_\alpha)\in\mc{D}_\alpha$ for all $X\in\mc{D}_\alpha$
on $M\setminus N_\beta$.
Indeed, $g(W_\alpha,X)=0$ because $W_\alpha=-W_{\beta}\in V_{\beta}^M=\mc{D}_\alpha^{\perp}$
and since $(w_\alpha^{-1}W_\alpha,1)=w_\alpha^{-1}(W_\alpha,w_\alpha)\in V_\alpha$,
we have
\[
V_\alpha\ni \nabla^{-1}_X (w_\alpha^{-1}W_\alpha,1)
&=(\nabla_X (w_\alpha^{-1}W_\alpha)+X,0+g(X,w_\alpha^{-1}W_\alpha))\\
&=(\nabla_X (w_\alpha^{-1}W_\alpha)+X,0).
\]
Because $X\in (V_{\beta}^M)^\perp$, it then follows that
\[
0&=h(\{(\nabla_X (w_\alpha^{-1}W_\alpha)+X,0)\}\times V_{\beta})
=g(\{\nabla_X (w_\alpha^{-1}W_\alpha)+X)\}\times V_{\beta}^M) \\
&=g(\{\nabla_X (w_\alpha^{-1}W_\alpha)\}\times V_{\beta}^M),
\]
i.e., $\nabla_X (w_\alpha^{-1}W_\alpha)\in (V_{\beta}^M)^\perp=\mc{D}_\alpha$.
\end{proof}

\begin{lemma}\label{le:intViM}
The distributions $V_1^M$ and $V_2^M$ are integrable
and their integral manifolds are totally geodesic.
\end{lemma}

\begin{proof}
Fix $\alpha=1,2$ and let $x\in M\setminus N_\alpha$.
Since $V_\alpha^M$ has constant rank around $x$,
the integrability of it in a neighborhood $U$
of $x$ which does not intersect $N_\alpha$,
is equivalent to the involutivity of $V_\alpha^M$ on $U$.

Thus take $X,Y\in\VF(U)$ which are tangent to $V_\alpha^M$.
Then there is a unique $s\in\Cinf(U)$ such that $(Y,s)\in V_\alpha$
on $U$.
But then
\[
V_\alpha\ni \nabla^{-1}_{X} (Y,s)=(\nabla_X Y+sX,X(s)+g(X,Y)),
\]
which implies that on $U$
\[
\nabla_X Y+sX\in V_\alpha^M.
\]
Since $X$ is also tangent to $V_\alpha^M$ on $U$,
we get that $\nabla_X Y\in V_\alpha^M$ on $U$ as well. 

But since $\nabla$ is torsion free
and since by the above $\nabla_X Y,\nabla_Y X\in V_\alpha^M$ on $U$,
one has that $[X,Y]=\nabla_X Y-\nabla_Y X\in V_\alpha^M$ on $U$,
i.e. $V_\alpha^M$ is involutive on $U$.

Moreover, if $O$ is an integral manifold of $V_\alpha^M$ through $y\in U$,
and if $X,Y\in\VF(O)$,
then on some neighborhood $U'\subset U$ of $y$ in $M$,
there are $\tilde{X},\tilde{Y}\in\VF(U')$ which restrict to $X,Y$ on $O$
and are tangent to $V_\alpha^M$ on $U'$.
Then $\nabla_{X} Y=\nabla_{\tilde{X}} \tilde{Y}$ on $O$
and by what was shown above, this is tangent to $V_\alpha^M$
i.e. tangent to $O$.
Thus $O$ is totally geodesic.

This proves that $V_\alpha^M$ is involutive on $M\setminus N_\alpha$
and that its integral manifolds are totally geodesic.
Since $N_2=\emptyset$, the only thing left
is to notice that by Lemma \ref{le:intDi}, $N_1$ is an integral manifold of $V_1^M$
because $\mc{D}_1|_y=V_1^M|_y$ for all $y\in N_1$.
\end{proof}

\begin{lemma}\label{le:par01}
For every $x\in M$ and every unit vector $u\in T|_x M$, one has
\[
(P^{\nabla^{-1}})_0^t(\gamma_u)(0,1)=(-\sinh(t)\dot{\gamma}_u(t),\cosh(t)),
\]
where $\gamma_u(t)=\exp_x(tu)$.
In particular, if $x\in N_1$ and $u\in V_2^M|_x$, $\n{u}_g=1$, then $\dot{\gamma}_u(t)\in V_1^M\cap V_2^M$
for all $t\neq 0$.
\end{lemma}

\begin{proof}
Let $(X(t),r(t)):=(-\sinh(t)\dot{\gamma}_u(t),\cosh(t))$. Then $(X(0),r(0))=(0,1)$ and the covariant derivative $\nabla^{-1}_{\dot{\gamma}_u(t)} (X,r)$ equals
\[
&\Big(\nabla_{\dot{\gamma}_u(t)} \big(-\sinh(t)\dot{\gamma}_u(t)\big)+\cosh(t)\dot{\gamma}_u(t),\dif{t}\cosh(t)+g(-\sinh(t)\dot{\gamma}_u(t),\dot{\gamma}_u(t))\Big) \\
&=\big(-\cosh(t)\dot{\gamma}_u(t)+\cosh(t)\dot{\gamma}_u(t),\sinh(t)-\sinh(t)\n{u}_g^2\big) \\
&=(0,0).
\]
This proves that $(X(t),r(t))=(P^{\nabla^{-1}})_0^t(\gamma_u)(0,1)$.

We prove the second claim.
Let $x\in N_1$ and $u\in V_2^M|_x$, $\n{u}_g=1$.
Since $V_2^M$ is integrable and its integral manifolds are totally geodesic by Lemma \ref{le:intViM},
it follows that $\dot{\gamma}_u(t)\in V_2^M$ for all $t$.
On the other hand, since $(0,1)\in V_1|_x$ by the definition of the set $N_1$,
we have $(P^{\nabla^{-1}})_0^t(\gamma_u)(0,1)\in V_1$ for all $t$,
i.e., $(-\sinh(t)\dot{\gamma}_u(t),\cosh(t))\in V_1$ for all $t$
and this implies that $-\sinh(t)\dot{\gamma}_u(t)\in V_1^M$ for all $t$.
Hence $\dot{\gamma}_u(t)\in V_1^M$ if $t\neq 0$.
\end{proof}

\begin{lemma}\label{le:warpedV2M}
\begin{itemize}
\item[(i)] Let $\{\alpha,\beta\}=\{1,2\}$.
Then if $x\notin M\backslash N_{\beta}$,
then there is an integral manifold $O_\alpha$
of $V_\alpha^M$ through $x$
such that $(O_\alpha,g|_{O_\alpha})$ is isometric to $(I\times M_\alpha,\diff s^2\oplus_{f_\alpha(s)} g_\alpha)$
where $I\subset\R$ is an open interval and $f_\alpha\in\Cinf(I)$ satisfies
$f_\alpha''-f_\alpha=0$.

\item[(ii)] If $x\in N_1$ then there exists an integral manifold $O_2$ of $V_2^M$ through $x$
and $(O_2,g|_{O_2})$ has constant curvature $-1$ if $\rank V_2^M\geq 2$.

\end{itemize}
\end{lemma}

\begin{proof}
(i)
Without loss of generality, one may assume that $\alpha=2$, $\beta=1$,
since the proof of the other case is completely symmetric.
Assume that $x\notin N_1$
and let $O_2$ be an integral manifold of $V_2^M$
through $x\in M$ such that $O_2\cap N_1=\emptyset$.
Clearly it is enough to assume that $\dim O_2\geq 2$.
In this case, the 1-dimensional integral manifolds of the distribution $\R W_2=V_1^M\cap V_2^M$ spanned by $W_2$ on $O_2$
are geodesics since they 
are (locally) the intersections of integral manifolds of $V_1^M$ and $V_2^M$,
which are totally geodesic by Lemma \ref{le:intViM}.
Moreover, integral manifolds of $\mc{D}_2$
are spherical by Lemma \ref{le:intDi}
and $T|_y O_2=\R W_2|_y\oplus \mc{D}_2|_y$ for all $y\in O_2$,
so $O_2$ is locally a warped product (see Theorem \ref{th:warped}) of the form
$(I\times M_2,\diff s^2\oplus_{f_2(s)} g_2)$ where $I\subset\R$ is an open interval
and $f_2\in\Cinf(I)$.
Moreover, for $y\in O_2$ and $X\in \mc{D}_2|_y=T|_y M_2$, we have
\[
g(W_2,W_2)X=B(X,W_2)W_2
=-R(X,W_2)W_2=\frac{H^{f_2}(W_2,W_2)}{f_2}X
=\frac{f''_2}{f_2}g(W_2,W_2)X,
\]
where the second equality follows from Lemma \ref{le:curvature}
and the third equality follows from \cite[Proposition 42, case (2)]{oneill83}, 
(here $H^{f_2}$ is the Hessian of $f_2$).
Taking any non-zero $X\in \mc{D}_2|_y$
(there exist one since $\rank\mc{D}_2\geq 1$)
and noticing that $W_2|_y\neq 0$ since $y\notin N_1$,
we get the claimed equation $f_2''=f_2$ for $f_2$.
This establishes the first part of the lemma.

(ii) Assume that $x\in N_1$.
Let $k:=\rank V_2^M=n-m$
and let $\epsilon>0$ be small enough such that
$\exp_x$ is defined and diffeomorphism from $B:=\{X\in V_2^M|_x\ |\ \n{X}_g<\epsilon\}$
onto its image and that $B\cap N_1=\{x\}$
(this is possible since $N_1$ is embedded in $M$).
Since integral manifolds of $V_2^M$ are totally geodesic,
it follows that $O_2:=\exp_x B$ is an integral manifold of $V_2^M$.

Suppose $u,X\in V_2^M|_x$ be such that $u\perp X$ and $\n{u}_g=1$
and define for $t\in [0,\epsilon[$,
\[
\gamma_u(t)&:=\exp_x(tu), \\
Y_{u,X}(t)&:=\sinh(t)P_0^t(\gamma_u)X.
\]
We claim that $Y_{u,X}$ is the Jacobi field $Y$ along $\gamma_u$
such that $Y(0)=0$, $\nabla_{u} Y(0)=X$.

The claims about initial values being obviously true for $Y_{u,X}$,
it remains to show that $Y_{u,X}$ satisfies the Jacobi-equation.
Notice that $\dot{\gamma}_u(t)\perp Y_{u,X}(t)$ for all $t$
and recall that by Lemmas \ref{le:V1McapV2M} and \ref{le:par01}, $\dot{\gamma}_u(t)\in V_1^M\cap V_2^M=\R W_1$ when $t\neq 0$.
Therefore Lemma \ref{le:curvature} implies that
\[
R(\dot{\gamma}_u(t),Y_{u,X}(t))\dot{\gamma}_u(t)
=-B(\dot{\gamma}_u(t),Y_{u,X}(t))\dot{\gamma}_u(t)
=g(\dot{\gamma}_u(t),\dot{\gamma}_u(t))Y_{u,X}(t)
=Y_{u,X}(t),
\]
while
\[
\nabla_{\dot{\gamma}_u(t)}\nabla_{\dot{\gamma}_u} Y_{u,X}(t)
=\nabla_{\dot{\gamma}_u(t)} (\cosh(t)P_0^t(\gamma_u)X)=\sinh(t)P_0^t(\gamma_u)X=Y_{u,X}(t).
\]
Hence the claim has been established.

Let $S^{k-1}$ be the $(k-1)$-dimensional  unit sphere $\{X\in V_2^M|_x\ |\ \n{X}_g=1\}$ of $V_2^M|_x$
and define
\[
F:(]0,\epsilon[\times S^{k-1},\diff s^2\oplus_{\sinh(s)} g|_{S^{k-1}})\to (M,g);
\quad F(s,X)=\exp_x(sX).
\]
Then $F$ is a diffeomorphism onto $O_2\backslash\{x\}=\exp_x (B\backslash \{0\})$
and
\[
F_*|_{(s,u)}(\alpha\doo_s+X)=\alpha\dot{\gamma}_u(s)+Y_{u,X}(s),
\]
whence
\[
\n{F_*|_{(s,u)}(\alpha\doo_s+X)}^2_{g}
=&\alpha^2\n{u}_g^2+\n{Y_{u,X}(s)}_g^2
=\alpha^2\n{u}_g^2+\sinh^2(s)\n{X}_g^2 \\
=&\n{\alpha\doo_s+X}^2_{\diff s^2\oplus_{\sinh(s)} g|_{S^{k-1}}}.
\]
This means that the mapping $F$ is an isometry from
$(]0,\epsilon[\times S^{k-1},\diff s^2\oplus_{\sinh(s)} g|_{S^{k-1}})$
onto $(O_2\backslash\{x\},g|_{O_2\backslash\{x\}})$.
Since the former Riemannian manifold has constant curvature $-1$,
if $k\geq 2$,
it follows that $(O_2,g|_{O_2})$ has constant curvature $-1$, if $k\geq 2$.
\end{proof}

We now finish the proof of Proposition \ref{pr:main-local}.

\begin{proof}[Proof of Proposition \ref{pr:main-local}]
First we consider the case where $k:=\rank V^M_2\geq 2$.
Let $x\in M$ be fixed.
Since integral manifolds of the constant rank distribution $\mc{D}_1$ are spherical and
those of $V_2^M=\mc{D}_1^\perp$ are totally geodesic,
Theorem \ref{th:warped}
implies that there is a neighborhood $U$ of $x$ in $M$
such that $(U,g|_U)$
is isometric to a warped product $(O_2\times M_1,g|_{O_2}\oplus_{f_1} g_1)$
where $O_2$ is an integral manifold of $V_2^M$ through $x$,
$M_1$ is an integral manifold of $\mc{D}_1$ through $x$ and $f_1\in\Cinf(O_2)$
(see Remark \ref{re:warped}).

First consider the situation where $x\in M\backslash N_1$.
By Lemma \ref{le:warpedV2M}, after shrinking $U$ if necessary around $x$,
$(O_2,g|_{O_2})$ is isometric to $(I\times M_2,\diff s^2\oplus_{f_2(s)} g_2)$
where $I\subset\R$ is an open interval and $f_2\in\Cinf(I)$.
We may also assume that $U\cap N_1=\emptyset$.
Hence, $(U,g|_U)$ is isometric to
\[
(I\times M_2\times M_1,(\diff s^2\oplus_{f_2(s)}g_2)\oplus_{f_1} g_1),
\]
where $f_1\in\Cinf(O_2)$, $f_2\in\Cinf(I)$ and $f_2''-f_2=0$.

We show that $f_1\in\Cinf(I)$ and $f_1''-f_1=0$.
Indeed, if $X\in W_1^\perp$, we have
\begin{align}\label{eq:f_1s}
X(f_1)=g(\nabla f_1,X)=-\frac{f_1}{w_1}g(W_1,X)=0,
\end{align}
which shows that $f_1\in\Cinf(I)$.
Moreover, for $y\in O_2$ and $X\in W_2^\perp|_y\cap T|_yO_2=\mc{D}_2|_y$, we have
(see Lemma \ref{le:curvature} and \cite{oneill83}, Proposition 42, case (2))
\begin{align}\label{eq:ode-f_1}
g(W_2,W_2)X=B(X,W_2)W_2
=-R(X,W_2)W_2=\frac{H_{f_1}(W_2,W_2)}{f_1}X
=\frac{f''_1}{f_1}g(W_2,W_2)X.
\end{align}
Because $\rank V_2^M\geq 2$, it follows that $\rank \mc{D}_2\geq 1$ on $U$.
Therefore, one may take above $X\neq 0$,
which implies that $f_1''=f_1$ and proves the claim.

We also know that
\[
w_\alpha^{-1} W_\alpha=-\frac{f_\alpha'}{f_\alpha}\doo_s,\quad \alpha=1,2,
\]
while
\[
0=&w_1^{-1}w_2^{-1}h((W_1,w_1),(W_2,w_2))
=h((w_1^{-1}W_1,1),(w_2^{-1}W_2,1))
=\frac{f_1'f_2'}{f_1f_2}-1.
\]
Writing
\[
f_\alpha(s)=A_\alpha\cosh(s)+B_\alpha\sinh(s),
\]
the above means that $A_1A_2-B_1B_2=0$.
Since $(A_i,B_i)\neq (0,0)$, $i=1,2$,
we may rescale, if necessary, the metrics of $(M_1,g_1)$ and $(M_2,g_2)$
so as to guarantee that either a) $A_1^2-B_1^2=+1$ and $A_2^2-B_2^2=-1$
or b) $A_1^2-B_1^2=-1$ and $A_2^2-B_2^2=+1$.
But since $h|_{V_2}$ is positive definite,
$\n{(W_2,w_2)}_{h}^2\geq 0$ and hence
\[
0\leq \n{(w_2^{-1}W_2,1)}_h^2=\left(\frac{f_2'(s)}{f_2(s)}\right)^2-1,
\]
which implies that $|A_2|\leq |B_2|$,
so only Case a) is possible.

It then follows easily that on $I$,
\[
f_1^2-f_2^2=1.
\]
Thus there is $s_0\in \R$ such that if one writes
$\tilde{I}=I-s_0$, then for all $\tilde{s}\in \tilde{I}$,
\[
f_1(\tilde{s}+s_0)&=\cosh(\tilde{s})=:\tilde{f}_1(\tilde{s}), \\
f_2(\tilde{s}+s_0)&=\sinh(\tilde{s})=:\tilde{f}_2(\tilde{s}).
\]
Since $(I\times M_2\times M_1,\diff s^2\oplus_{f_2(s)} g_2\oplus_{f_1(s)} g_1)$
is isometric to $(\tilde{I}\times M_1\times M_2,\diff s^2\oplus_{\tilde{f}_1(s)} g_1\oplus_{\tilde{f}_2(s)} g_2)$,
we have proved {\rm (LW2)} of Proposition~\ref{pr:main-local} when $\rank V_2^M\geq 2$.

Next we consider the case where $x\in N_1$.
In this situation, \ref{le:warpedV2M} case (ii) implies (after maybe shrinking $U$ around $x$)
that $(O_2,g|_{O_2})$ is isometric to an open subset of $(\HH^k,{\bf g}_{k;-1})$.
Moreover, by the proof of \ref{le:warpedV2M} case (ii),
we may assume that $(O_2\backslash\{x\},g|_{O_2\backslash\{x\}})$
is isometric to $(]0,\epsilon[\times S^{k-1},\diff s^2\oplus_{\sinh(s)} g|_{S^{k-1}})$
by the map $F$ introduced there. 

We show that $\tilde{f}_1:=f_1\circ F\in\Cinf(]0,\epsilon[)$ and $\tilde{f}_1''-\tilde{f}_1=0$.
Indeed, if $X\in W_1^\perp$, then $X(f_1)=0$ by \eqref{eq:f_1s}.
Hence, in particular, if $\tilde{X}\in T|_{F^{-1}(y)} S^{k-1}$ for $y\in U\backslash\{x\}$,
then $\tilde{X}\perp \partial_s$ i.e. $F_*(\tilde{X})\perp W_1$,
and hence $\tilde{X}(\tilde{f}_1)=0$. This shows that $\tilde{f}_1$
is constant on each set $\{s\}\times S^{k-1}$, $s\in ]0,\epsilon[$,
since $S^{k-1}$ is connected
and thus $\tilde{f}_1\in \Cinf(]0,\epsilon[)$.
From Eq. \eqref{eq:ode-f_1} one then infers that $\tilde{f}_1''-\tilde{f}_1=0$.

Thus for some $A_1,B_1\in\R$,
\[
\tilde{f}_1(s)=A_1\cosh(s)+B_1\sinh(s).
\]
Recall that $-\frac{\nabla f_1}{f_1}\doo_s=w_1^{-1}W_1$ where
\[
(0,1)=(W_1,w_1)+(W_2,w_2),
\]
with $(W_1,w_1)\in V_1$, $(W_2,w_2)\in V_2$.
Since $\n{(W_2,w_2)}_h^2\geq 0$, $\n{(0,1)}_h^2=-1$
and $\nabla f_1=\tilde{f}_1'F_*(\partial_s)$,
it follows that 
\begin{align}\label{eq:B_1leqA_1}
\left(\frac{\tilde{f}_1'}{\tilde{f}_1}\right)^2-1=\n{(w_1^{-1}W_1,1)}_h^2<0,
\end{align}
and hence $|B_1|<|A_1|$.
Then, one can normalize $A_1,B_1$ such that $A_1^2-B_1^2=-1$
and by eventually replacing, as before, $s$ by $s+s_0$, for some $s_0\in\R$
(these operations just rescale the metric $g_1$ by a constant),
one gets $\tilde{f}_1(s)=\cosh(s)$.

If $d(y):=d(x,y)$ is the distance function of $(O_2,g|_{O_2})$ from $x$,
then clearly $s=d(F(s,u))$ for $(s,u)\in ]0,\epsilon[\times S^{k-1}$,
which implies that $f_1(y)=\cosh(d(y))$ for all $y\in O_2$.
Thus we have arrived at {\rm (LW3)} when $\rank V_2^M\geq 2$.

It remains to provide an argument for the case $\rank V_2^M=1$.
Let $x\in M$. Then by Lemma \ref{le:warpedV2M} case (i) with $\alpha=1$,
one gets that $(O_1,g|_{O_1})$ is isomorphic to a warped product $(I\times M_1,\diff s^2\oplus_{f_1(s)} g_1)$
where $f_1\in\Cinf(I)$ satisfies $f_1''-f_1=0$.
Then exactly the same argument as above, replacing $\tilde{f}_1$ by $f_1$,
leads to \eqref{eq:B_1leqA_1}
and to the conclusion that we may take $f_1(s)=\cosh(s)$
(after scaling the metric $g_1$ by a constant).
Hence we have {\rm (LW1)} and the proof of Proposition \ref{pr:main-local} complete.

\end{proof}

\begin{proof}[Proof of Theorem \ref{th:main}]
Let us now assume that $(M,g)$ is complete and simply connected
and use the notation of the above proof of Proposition \ref{pr:main-local}. 

By Theorem \ref{th:warped}, we have in the proof of Proposition \ref{pr:main-local}
that $(O_2,g|_{O_2})$ and $(M_1,g_1)$ are complete
and simply connected.
Since $N_1\neq\emptyset$, we may further assume that $x\in N_1\cap O_2$.

Then if $k:=\rank V_2^M\geq 2$, the argument leading to {\rm (LW3)}
shows, since one may take $\epsilon=+\infty$ there, that
$(O_2,g|_{O_2})$ is isometric to $(\HH^k,{\bf g}_{k;-1})$
and that $f_1$ can be chosen to be $\cosh(d(\cdot))$
where $d$ is the distance function on $(\HH^k,{\bf g}_{k;-1})$
from the point corresponding to $x$.
This proves {\rm (WP2)}.

If $k=\rank V_2^M=1$, then in the argument leading to {\rm (LW1)},
one may take $I=\R$
and hence we have {\rm (WP1)}.
This completes the proof of Theorem \ref{th:main}.

\end{proof}

\subsubsection{Case $V_1\cap V_2\neq\{0\}$}

In this case, $\dim (V_1\cap V_2)=1$
and $V_1\cap V_2$ is lightlike and invariant
by $\mcH^{-1}$ since $V_1$ and $V_2$ are.
Therefore, at every point $x\in M$,
there existsa tangent vector $L|_x\in T|_x M$ such that $V_1|_x\cap V_2|_x=\R(L|_x,1)$.
In this way, we may choose $L|_x$ locally
such that $L:=(x\mapsto L|_x)$ becomes a smooth locally defined
vector field on $M$ and if $M$ is simply connected,
$L$ can be chosen to be globally defined.

Since $(L,1)$ is lightlike vector in $T|_x M\oplus\R$,
we have $0=\n{(L,1)}_h^2=\n{L}_g^2-1$
i.e. $L$ is a unit vector field.

\begin{lemma}\label{le:nablaX-L}
For all $X\in TM$, we have that
$\nabla_X L=-X+g(X,L)L$.
\end{lemma}

\begin{proof}
Since $V_1\cap V_2=\R(L,1)$ and because $V_1\cap V_2$
is invariant under $\mcH^{-1}$,
we get that $V_1\cap V_2$ is invariant under parallel transport with respect to $\nabla^{-1}$.
This is equivalent to the fact that for any $X\in TM$ there is $\alpha(X)\in\R$ such that
$\nabla^{-1}_X (L,1)=\alpha(X)(L,1)$, i.e.,
\[
(\nabla_X L+X,g(X,L))=\alpha(X)(L,1),
\]
from which one gets $\alpha(X)=g(X,L)$ and thus $\nabla_X L+X=g(X,L)L$.
\end{proof}

\begin{lemma}\label{le:warpedL}
The vector field $L$ is geodesic,
the distribution $L^\perp$ is integrable
and its integral manifolds are spherical.
\end{lemma}

\begin{proof}
By Lemma \ref{le:nablaX-L} and the fact that $\n{L}_g=1$,
we get $\nabla_L L=-L+g(L,L)L=-L+L=0$ so $L$ is a geodesic vector field.

Let us prove the integrability of $L^\perp$.
If $X,Y\in L^\perp$, then
\[
g([X,Y],L)=&g(\nabla_X Y-\nabla_Y X,L)
=-g(Y,\nabla_X L)+g(X,\nabla_Y L) \\
=&-g(Y,-X+g(X,L)L)+g(X,-Y+g(Y,L)L) \\
=&g(Y,X)-g(X,Y)=0,
\]
i.e., $[X,Y]\in L^\perp$. This proves that $L^\perp$ is involutive and hence integrable.

Let $O$ be an integral manifold of $L^\perp$. 
If $X,Y\in\VF(O)$, we get
\[
g(\nabla_X Y,L)=-g(Y,\nabla_X L)=-g(Y,-X+g(X,L)L)=g(X,Y),
\]
so the second fundamental form $\II_O$ of $O$
is given by
\[
\II_O(X,Y)=g(X,Y)L,\quad X,Y\in T|_x M,\ x\in O,
\]
which means that $O$ is umbilical.

To show that $O$ is spherical, we need to show that $\nabla_X L\in L^\perp$ for all $X\in TO$.
But this is clear since $0=Xg(L,L)=2g(\nabla_X L,L)$.
This completes the proof.
\end{proof}

We now finish the proof of Proposition \ref{pr:main-local} in this case.
By the previous lemma and Theorem \ref{th:warped},
it follows that
locally $(M,g)$ is isometric
to a warped product $(I\times M_1,\diff s^2\oplus_f g_1)$
for some interval $I\subset\R$ and $f\in\Cinf(I)$.
If $(M,g)$ is complete and simply connected, then $I=\R$.
Moreover, one has
\[
\frac{f'}{f} X=\nabla_X L=-X+g(X,L)L=-X,
\]
for any $X\in L^\perp$. It follows that $f(s)=Ce^{-s}$ for some $C\neq 0$.
By rescaling the metric $g_1$ by a constant, we may assume that $C=1$.

\subsection{Proof of Proposition \ref{pr:converse}}

\subsubsection{Case $V_1\cap V_2\neq\{0\}$:}

Suppose $(M,g)=(I\times M_1,\diff s^2\oplus_{e^{-s}} g_1)$.
Let $L:=\doo_s$, $f(s)=e^{-s}$ and compute that for all $Y\in L^\perp$,
\[
\nabla_Y L=\frac{f'}{f} Y=-Y,
\]
and $\nabla_L L=0$, so for every $X\in TM$
\[
\nabla_X L=-X+g(X,L)L.
\]
Define a one-dimensional subbundle of $\pi_{TM\oplus\R}$
whose fibers are $V_1:=\R (L,1)$.
Then $V_1$ is light-like and for every $X\in TM$,
\[
\nabla^{-1}_X (L,1)=(\nabla_X L+X,g(X,L))=g(X,L)(L,1),
\]
which shows that $V_1$ is invariant by parallel transport with respect to $\nabla^{-1}$. In particular, $V_1$ is invariant with respect to $\mcH^{-1}$ and therefore $\mcH^{-1}$ is reducible.
This proves that $\mc{H}^{-1}$ is reducible if $(M,g)$ is of the form {\rm (LW1)}.

\subsubsection{Case $V_1\cap V_2=\{0\}$:}
Assume first that $(M,g)=(I\times M_2\times M_1,\diff s^2\oplus_{\sinh(s)}g_2 \oplus_{\cosh(s)} g_1)$.
Here $I\subset\R$ is an interval not containing $0$.
Define for every $x=(s,x_2,x_1)\in M$,
\[
V_1|_x:=\R(W_1|_x,w_1(x))\oplus (T|_{x_1} M_1\times\{0\})\subset T|_x M\oplus\R,
\]
where
\[
(W_1|_x,w_1(x))&:=\cosh(s)(-\sinh(s)\doo_s,\cosh(s)).
\]
We prove that $V_1$ is invariant under $\mcH^{-1}$.

Indeed, let $X\in T|_{x_1} M_1$, $Y\in\VF(M_1)$ and $Z\in T|_{(s,x_2)} (I\times M_2)$.
Then
\[
\nabla_X^{-1} (Y,0)&=(\nabla_X Y,g(X,Y))
=\big(\nabla^{g_1}_X Y-g(X,Y)\tanh(s)\doo_s,g(X,Y)\big)  \\
&=(\nabla^{g_1}_X Y,0)+\frac{g(X,Y)}{\cosh^2(s)}(W_1|_x,w_1(x))\in V_1|_{x}, \\
\nabla_Z^{-1} (Y,0)&=(\nabla_Z Y,g(Z,Y))=\big(g(Z,\doo_s)\tanh(s)Y,0\big)\in V_1|_{x},
\]
and if $U\in T|_{x_2} M_2$,
\[
\nabla^{-1}_X (W_1,w_1)&=(-\sinh^2(s)X+w_1 X,0)=(X,0)\in V_1|_x, \\
\nabla^{-1}_U (W_1,w_1)&=(-\cosh^2(s)U+w_1 U,0)=(0,0)\in V_1|_x, \\
\nabla^{-1}_{\doo_s} (W_1,w_1)&=\big((-\sinh^2(s)-\cosh^2(s))\doo_s+w_1\doo_s,2\cosh(s)\sinh(s)-\sinh(s)\cosh(s)\big) \\
&=\sinh(s)(-\sinh(s)\doo_s,\cosh(s))=\tanh(s)(W_1,w_1)\in V_1|_x.
\]

These formulas show that
for all $X\in TM$ and $Y\in \Gamma(\pi_{V_1})$,
one has $\nabla^{-1}_X Y\in\Gamma(\pi_{V_1})$.
Thus $V_1$ is invariant under parallel transport with respect to $\nabla^{-1}$
and therefore it is invariant under $\mcH^{-1}$.
This proves that $\mc{H}^{-1}$ is reducible if $(M,g)$ is of the form {\rm (LW2)}.

Consider then the case where $(M,g)=(O\times M_1,{\bf g}_{k;-1} \oplus_{\cosh(d(\cdot))}\oplus g_1)$,
where $O$ is a normal neighbourhood of a point $x_0\in \HH^k$
and  $d(x)=d(x,x_0)$ is the distance function
from $x_0$ in $\HH^k$.

Observe that $(O\backslash \{x_0\},{\bf g}_{k;-1})$
is isometric to $(]0,\epsilon[\times S^{k-1},\diff s^2\oplus_{\sinh(s)} {\bf g}_{k-1;1})$,
for $\epsilon>0$ or $\epsilon=+\infty$
and $d(x)=s$ if $x\neq x_0$ corresponds to $(s,y)\in ]0,\epsilon[\times S^{k-1}$.
Choosing above $(M_2,g_2)=(S^{k-1},{\bf g}_{k-1;1})$,
we conclude that $\mcH^{-1}|_{]0,\epsilon[\times M_2\times M_1}$ is reducible
and hence by continuity, $\mcH^{-1}$ is reducible on $(M,g)$.
We conclude that if $(M,g)$ is of the form {\rm (LW3)}, then $\mc{H}^{-1}$ is reducible
and hence completes the proof of Proposition~\ref{pr:converse}.

The proof of sufficiency of Theorem \ref{th:main} now follows immediately from
the above, since one could take $O=\HH^k$, the exponential map of $\HH^k$ at $x_0$
being a diffeomorphism of $T|_{x_0} \HH^k$ onto $\HH^k$.



\begin{thebibliography}{5}\setlength{\itemsep}{-0.5mm}
\bibitem{ACL} Alouges, F., Chitour Y., Long, R. \emph{A motion planning algorithm for the rolling-body problem},
IEEE Trans. on Robotics 26 (2010), no. 5, 827--836.


\bibitem{agrachev99} Agrachev A., Sachkov Y., \emph{An Intrinsic Approach to the Control of Rolling Bodies},
Proceedings of the Conference on Decision and Control, Phoenix, 1999, pp. 431 - 435, vol.1.

\bibitem{agrachev04} Agrachev, A., Sachkov, Y., \emph{Control Theory from the Geometric Viewpoint},
Encyclopaedia of Mathematical Sciences, 87. Control Theory and Optimization, II. Springer-Verlag, Berlin, 2004.


\bibitem{berger} Berger, M., \emph{Sur les groupes d'holonomie homog\`ene des vari\'et\'es \`a connexion affine et des vari\'et\'es riemanniennes}, Bulletin de la Soci\'et\'e Math\'ematique de France 83 (1955), 279--330.


\bibitem{bryant06} Bryant, R., \emph{Geometry of Manifolds with Special Holonomy: "100 Years of Holonomy"},
Contemporary Mathematics, Volume 395, 2006.

\bibitem{bryant-hsu} Bryant, R. and Hsu, L., \emph{Rigidity of integral curves of rank 2 distributions}, Invent. Math. 114 (1993), no. 2, 435--461.



\bibitem{cartan25} Cartan, \'E., \emph{La g\'eom\'etrie des espaces de Riemann}, M\'emorial des sciences math\'ematiques,
fascicule 9 (1925), 1--61.

\bibitem{chelouah01} Chelouah, A. and Chitour, Y., \emph{On the controllability and trajectories generation of rolling surfaces}, Forum Math. 15 (2003) 727--758.

\bibitem{arxiv} Chitour, Y., Kokkonen, P., \emph{Rolling Manifolds: Intrinsic Formulation and Controllability},
arXiv:1011.2925v2, 2011.

\bibitem{CK} Chitour, Y., Kokkonen, P., \emph{Rolling Manifolds on Space Forms.} Submitted.






\bibitem{dSO} Di Scala, A.J., Olmos, C., \emph{The geometry of homogeneous submanifolds of hyperbolic space.} Math. Z. 237(1) (2001) 199--209.

\bibitem{norway} Godoy Molina, M., Grong, E., Markina, I., Leite, F.,
\emph{An intrinsic formulation of the rolling manifolds problem}. To appear J. Dyn. Control Syst.

\bibitem{GG} Godoy Molina, M., Grong, E.,
\emph{Geometric conditions for the existence of an intrinsic rolling}. arXiv:1111.0752


\bibitem {hiepko} Hiepko, S.	
\emph{Eine innere Kennzeichnung der verzerrten Produkte},
Math. Ann. 241 (1979), no. 3, 209--215. 


\bibitem{jurd} Jurdjevic, V. \emph{Geometric control theory}, Cambridge Studies in Advanced Mathematics, 52. Cambridge University Press, Cambridge, 1997.

\bibitem{kobayashi63} Kobayashi, S., Nomizu, K., \emph{Foundations of Differential Geometry, Vol. I}, Wiley-Interscience, 1996.



\bibitem{marigo-bicchi} Marigo, A. and Bicchi A., \emph{Rolling bodies with regular surface: controllability theory and applications}, IEEE Trans. Automat. Control 45 (2000), no. 9, 1586--1599.


 


\bibitem{murray-sastry} Murray, R., Li, Z. and Sastry, S. \emph{A mathematical introduction to robotic manipulation}, CRC Press, Boca Raton, FL, 1994.


\bibitem{oneill83} O'Neill, B., \emph{Semi-Riemannian Geometry with Applications to Relativity}, Academic Press, 1983





\bibitem{sakai91} Sakai, T., \emph{Riemannian Geometry},
Translations of Mathematical Monographs, 149. American Mathematical Society, Providence, RI, 1996.


\bibitem{sharpe97} Sharpe, R.W., \emph{Differential Geometry: Cartan's Generalization of Klein's Erlangen Program},
Graduate Texts in Mathematics, 166. Springer-Verlag, New York, 1997.




\bibitem{wu64} Wu, H., \emph{On the de Rham decomposition theorem}, Illinois J. Math. 8, (1964), 291--311

\end{thebibliography}
\end{document}